\documentclass[11pt]{amsart}
\usepackage{amssymb}
\usepackage{amsfonts}

\usepackage{bbm}
\usepackage{mathrsfs}
\usepackage{color}

\usepackage{array}

\usepackage{amsfonts,amsmath, amssymb}

\newcommand{\bea}{\begin{eqnarray}}

\newcommand{\eea}{\end{eqnarray}}

\newcommand{\be}{\begin {equation}}

\newcommand{\ee}{\end{equation}}

\newtheorem{theorem}{Theorem}[section]

\newtheorem{corollary}[theorem]{Corollary}
\newtheorem{lemma}[theorem]{Lemma}

\newtheorem{remark}[theorem]{Remark}

\newtheorem{question}[theorem]{Question}

\begin{document}

\title{ $p$-adic L-functions and Classical Congruences}

\author {Xianzu Lin }

\date{ }
\maketitle
   {\small \it College of Mathematics and Computer Science, Fujian Normal University, }\\
    \   {\small \it Fuzhou, {\rm 350108}, China;}\\
      \              {\small \it Email: linxianzu@126.com}

\begin{abstract}
In this paper, using $p$-adic analysis and  $p$-adic L-functions,
we show how to extend classical congruences (due to Wilson, Gauss,
Dirichlet, Jacobi, Wolstenholme, Glaisher, Morley, Lemher and
other people) to modulo $p^k$ for any $k>0$.
\end{abstract}



Keywords:Congruence, $p$-adic L-function, $p$-adic logarithm,
Bernoulli number.


Mathematics Subject Classification 2010: Primary 11A07; secondary
11B68,11B65,11S40.

\section{Introduction}
Let $p$ be an odd prime.  A famous congruence due to Wilson (according to Waring) states that \cite[p.68]{HW}:
   \begin{equation}\label{fmain13}  (p-1)!\equiv-1\  (mod\ p).\end{equation}
  This congruence was first formulated by Waring in
  1770. The first proof was given by Lagrange in 1771. In 1828, Dirichlet proved another related congruence: \begin{equation} (\frac{p-1}{2})!\equiv(-1)^N\  (mod\ p),\end{equation}
  where $4|p-3$ and $N$ is the number of quadratic nonresidues less than $p/2$. In 1900 Glaisher \cite{gl1} extended Wilson's theorem as follow:
  \begin{equation}\label{fmain13} (p-1)!\equiv pB_{p-1}-p\  (mod\ p^2),\end{equation} where $B_n$ is the Bernoulli number.
In 2000,  Sun  \cite{sun7} went one step
further by showing
 \begin{equation}\label{fmain14} (p-1)!\equiv\frac{pB_{2p-2}}{2p-2}-\frac{pB_{p-1}}{p-1}-\frac{1}{2}(\frac{pB_{p-1}}{p-1})^2\  (mod\ p^3).\end{equation}

Assume that $p\equiv1(mod\ 4)$, then by Fermat's two square theorem, we have $p=a^2+b^2$, where $a$ can be uniquely determined by requiring $a\equiv1\ (mod\ 4)$.
 Another famous congruence, due to Gauss(1828) states that   \begin{equation} \label{fmain1fd} {(p-1)/2\choose (p-1)/4}\equiv2a\ (mod\ p).\end{equation}
The following extension of Gauss's congruence was first
conjectured by Beukers \cite{be} and proved by Chowla, Dwork,
Evans\cite{cde}:
 \begin{equation} \label{fmain1fd4} {(p-1)/2\choose (p-1)/4}\equiv\Big(1+\frac{1}{2}pq_p(2)\Big)\Big(2a-\frac{p}{2a}\Big)\ (mod\ p^2).\end{equation}
 In 1837, Jacobi proved a congruence analogous to (\ref{fmain1fd}): \begin{equation} \label{fmain1fd1} {2(p-1)/3\choose (p-1)/3}\equiv-r\ (mod\ p),\end{equation}
 where $p\equiv1(mod\ 6)$, $4p=r^2+27s^2$, and $r\equiv1(mod\ 3)$.
Evans and Yeung \cite{ye} independently extended Jacobi'
congruence to modulo $p^2$ as follow:
\begin{equation} \label{fmain1fd12} {2(p-1)/3\choose (p-1)/3}\equiv-r+\frac{p}{r}\ (mod\ p^2).\end{equation}
In 2010 Cosgrave and Dilcher further extended Gauss' and Jacobi'
congruences to modulo $p^3$ \cite{cd} as follows:
{\setlength{\arraycolsep}{0pt}
\begin{eqnarray}\label{fmain2655dsf}
&&{(p-1)/2\choose (p-1)/4}\equiv
 \Big(2a-\frac{p}{2a}-\frac{p^2}{8a^2}\Big)\cdot
\\
&&\Big(1+\frac{1}{2}q_p(2)p-\frac{1}{8}q_p(2)^2p^2+\frac{1}{4}E_{p-3}p^2\Big)
\ (mod\ p^3) \nonumber
\end{eqnarray}
}\begin{equation} \label{fmain1fd126} {2(p-1)/3\choose
(p-1)/3}\equiv\Big(-r+\frac{p}{r}+\frac{p^2}{r^3}\Big)\Big(1+\frac{1}{6}B_{p-2}(\frac{1}{3})p^2\Big)\
(mod\ p^3).\end{equation} Later, (\ref{fmain2655dsf}) and
(\ref{fmain1fd126}) were extended to cover similar binomial
coefficients by Al-Shaghay and Dilcher\cite{ad}.

In 1852, Wolstenholme proved his famous theorem\cite{wo} which
states that if $p\geq5$ is a prime, then
\begin{equation}\label{fmain1} \sum_{k=1}^{p-1}\frac{1}{k}\equiv0\  (mod\ p^2).\end{equation}

In the same paper, Wolstenholme also proved the following
congruence:
\begin{equation}\label{fmainn1} \sum_{k=1}^{p-1}\frac{1}{k^2}\equiv0\  (mod\ p).\end{equation}

In 1900, Glaisher  gave the following generalizations of
Wolstenholme's theorem \cite{gl1,gl2}:

\begin{equation}\sum_{k=1}^{p-1}\frac{1}{k^m}\equiv \begin{cases} \frac{m}{m+1}pB_{p-m-1}\  \  \  \ \  \ \  \  \ (mod\ p^2)& \text{if}\ m\ is\ even \\ & \\
 \frac{-m(m+1)}{2(m+2)}p^2B_{p-m-2}\  \  (mod\ p^3)& \text{if}\ m\ is\ odd \end{cases}\end{equation}

where $m>0$ and $p\geq m+3$,

Since then, the above congruences have been extended by several
authors  to modulo $p^k$ \cite{gl1,gl2,le,sun7,ta}, and the
multiple harmonic sums\cite{ho,ppt,ro,zhao3,zhao4,zhoucai}.

Another form of Wolstenholme's theorem, which can be easily deduced from (\ref{fmain1}) is that:

\begin{equation}\label{fmain2} {2p-1\choose p-1}\equiv1\  (mod\ p^3).\end{equation}

Many extensions of  (\ref{fmain2}) to modulo $p^k$ in terms of
Bernoulli numbers or harmonic sums have been obtained in
\cite{gl0,gl1,ht,me,ro,ta}. We refer the readers to \cite{me0} for
more references about various generalizations of Wolstenholme's
theorem.

Another classical congruence for binomial coefficient due to
Morley (1895) \cite{mo} is:
\begin{equation}\label{fmain3} {p-1\choose (p-1)/2}\equiv(-1)^{(p-1)/2}4^{p-1}\  (mod\ p^3).\end{equation}

For integer $a$ with $(a,p)=1$, set $q_p(a)=\frac{a^{p-1}-1}{p}$.
In 1938, Lemher \cite{le} proved the following four congruences:
\begin{equation}\label{fmain4y} \sum_{k=1}^{[p/2]}\frac{1}{k}\equiv-2q_p(2)+pq_p(2)^2\  (mod\ p^2);\end{equation}
\begin{equation}\label{fmain5y} \sum_{k=1}^{[p/3]}\frac{1}{p-3k}\equiv\frac{1}{2}q_p(3)-\frac{1}{4}pq_p(3)^2\  (mod\ p^2);\end{equation}
\begin{equation}\label{fmain6y} \sum_{k=1}^{[p/4]}\frac{1}{p-4k}\equiv\frac{3}{4}q_2-\frac{3}{8}pq_p(2)^2\  (mod\ p^2);\end{equation}
\begin{equation}\label{fmain47y} \sum_{k=1}^{[p/6]}\frac{1}{p-6k}\equiv\frac{1}{4}q_p(3)+\frac{1}{3}q_p(2)-\frac{1}{8}pq_p(3)^2-\frac{1}{6}pq_p(2)^2\big)\  (mod\ p^2),\end{equation}
and used them to derive congruences about ${p-1\choose [p/m]}$ for
$m=2, 3, 4$ or $6$. In\cite{ca,pan,sun7,sun8}, Morley's and
Lemher's congruences are extended to congruences for ${p-1\choose [p/m]}$ and
$\sum_{k=1}^{[p/m]}\frac{1}{k^n}$ modulo $p^k$.
\begin{remark}
Congruences (\ref{fmain4y}) and (\ref{fmain6y}) modulo $p$ were
given by Glaisher \cite{gl2}, while congruence (\ref{fmain4y}) and
(\ref{fmain5y}) modulo $p$ were given proved by Lerch \cite{ler}.
\end{remark}

There arises the following question:
\begin{question}Can we extend the above congruences to modulo $p^k$ for arbitrarily large $k$ ?\end{question}

The first breakthrough in this direction is due to Washington
\cite{wa}, who gave an explicit $p$-adic expansion of the sums
 \begin{equation}\label{fmain10}\sum_{k=1 \atop (k,np)=1}^{np}\frac{1}{k^m}\end{equation}
 in terms of $p$-adic $L$-functions. Washington' expansions, together with Kummer's congruences for $p$-adic $L$-functions,
 immediately imply mod $p^k$ evaluations of $\sum_{k=1}^{p-1}\frac{1}{k^m}$ for arbitrarily large $k$.

In this paper, we give $p$-adic expansions for the sums of
 Lemher's type
 \begin{equation}\label{fmain11}\sum_{k=1 \atop (k,np)=1}^{[np/r]}\frac{1}{k^m},\end{equation} where
 $(r,np)=1$. It turn outs that many binomial coefficients also admit nice expansions in terms of $p$-adic
$L$-functions. As applications,  we can extend all the congruences
mentioned above to modulo $p^k$ for arbitrarily large $k$.

This paper is structured as follows: In Section 2, we give
preliminaries that will be used throughout this paper. In Section
3, we give a review of Washington's
 $p$-adic expansion of the power sums and its applications. In Section 4, we give a similar  $p$-adic expansion
for the sums of Lemher's type and derive many corollaries. Sections 5 and 6 are devoted to extensions  of Gauss's and Jacobi's congruences, and Wilson's theorem respectively.

\section{preliminaries}
The Bernoulli numbers $B_n$ and the Bernoulli polynomials $B_n(x)$
are defined respectively by
\begin{equation}\label{fmain21} \frac{z}{e^z-1}=\sum_{n=0}^{\infty}B_nz^n/n!;\end{equation}
\begin{equation} \label{fmain22}\frac{ze^{xz}}{e^z-1}=\sum_{n=0}^{\infty}B_n(x)z^n/n!.\end{equation}
Thus $B_0(x)=1$, $B_1(x)=x-\tfrac{1}{2}$, $B_2(x)=x^2-x+\tfrac{1}{6}$, $B_3(x)=x^3-\tfrac{3}{2}x^2+\tfrac{1}{2}x$, $B_4(x)=x^4-2x^3+x^2-\tfrac{1}{30}$, etc.

From the above definitions, we have
\begin{equation} B_n(x)=\sum_{r=0}^{n}{n\choose r}B_rx^{n-r}.\end{equation}
In particular, $B_n(0)=B_n.$ Note that  $B_n=0$ whenever $n>1$ is
odd.

For a Dirichlet character $\chi$ modulo $m$, the generalized
Bernoulli numbers $B_{n,\chi}$ are defined by
\begin{equation}\label{fmain29} \sum_{a=1}^{m}\frac{\chi(a)ze^{az}}{e^{mz}-1}=\sum_{n=0}^{\infty}B_{n,\chi}z^n/n!.\end{equation}

 From the definitions, we have
\begin{equation}\label{fmain30} B_{n,\chi}=m^{n-1}\sum_{a=1}^{m}\chi(a)B_{n}(\tfrac{a}{m}).\end{equation}

For a Dirichlet character $\chi$ modulo $m$ and a positive integer
$d$, let $\chi'$ be  the character modulo $md$ induced by  $\chi$
. Then we have
\begin{equation}\label{fmain31}
B_{n,\chi'}=B_{n,\chi}\prod_{p\mid d, p\
prime}(1-\chi(p)p^{n-1}).\end{equation}

We need the  following  identity of power sums due to Szmidt,
Urbanowicz and Zagier \cite{suz}
   \begin{lemma}\label{lemma1}Let  $\chi$ be a Dirichlet character modulo $d$ and $N$ be a multiple of $d$. Let $m$ and $r$ be positive
   integers, with $(r,N)=1$. Then
$$mr^{m-1}\sum_{n=1}^{[N/r]}\chi(n)n^{m-1}=-B_{m,\chi}+\frac{\overline{\chi}(r)}{\varphi(r)}\sum_\psi\overline{\psi}(-N)B_{m,\chi\psi}(N).$$
   \end{lemma}

Now we recall definition and basic properties of $p$-adic
$L$-functions and refer the readers to \cite{wa0} for more
details.

Throughout this paper, $p$ denotes an odd prime, and
$\mathbb{Z}_p$ and $\mathbb{Z}_p^*$ denote the ring of $p$-adic
integers and the group of invertible $p$-adic integers
respectively. The $p$-adic-valued Teichm\"{u}ller character
$\omega$ is defined as follows:

For an integer $a$ with $(a,p)=1$, $\omega(a)\in\mathbb{Z}_p$ is the
$p-1$-st root of unit satisfying $\omega(a)\equiv a  (mod\ p)$. Set $\langle a \rangle=\omega(a)^{-1}a $

The $p$-adic exponential and logarithm functions are defined
respectively by
 \begin{equation}\label{fmain311} exp(s)=\sum_{n=0}^{\infty}s^n/n!,\end{equation}
\begin{equation} \label{fmain312}log_p(1+s)=\sum_{n=0}^{\infty}(-1)^{n+1}s^n/n,\end{equation}
for $s\in p\mathbb{Z}_p$. As usual, we have $ exp(log_p(1+s))=1+s$
and $log_p(exp(s))=s$, and
\begin{equation}\label{fmain2659xsws}
log_p(1+s)+log_p(1+t)=log_p((1+s)(1+t)),\end{equation} for $s,t\in
p\mathbb{Z}_p$.

Let  $\chi$ be a primitive Dirichlet character modulo $d$ and let
$D$ be any multiple of $p$ and $d$. The $p$-adic $L$-function
$\chi$ is defined by:
\begin{equation}\label{fmain32} L_{p}(s,\chi)=\frac{1}{D}\frac{1}{s-1}\sum_{a=1 \atop (a,p)=1}^{D}\chi(a)\langle a \rangle^{1-s} \sum_{n=0}^{\infty}{1-s \choose n}(B_{n})(\frac{D}{a})^n,\end{equation}
where $s\in\mathbb{Z}_p$ and \begin{equation}\label{fmain3211} \langle a \rangle^{1-s}=exp((1-s)log_p(\langle a \rangle))= \sum_{n=0}^{\infty}{1-s \choose n}(\langle a \rangle-1)^n.\end{equation}

From the definition, we have, for $n\geq1$,
\begin{equation}\label{fmain321} L_{p}(1-n,\chi)=-(1-\chi \omega^{-n}(p)p^{n-1})\frac{B_{n,\chi \omega^{-n}}}{n}.\end{equation}
By (\ref{fmain321}) it is easy to see that $L_{p}(s,\chi)$ is
identically zero if $\chi(-1)=-1$. We note that $L_{p}(s,\chi)$ is analytic if $\chi\neq\textbf{1}$, and  $L_{p}(s,\textbf{1})$ is analytic except for a pole at $s=1$  with residue $(1-1/p)$.
\begin{lemma}\label{lemma49} $p(1-s)L_{p}(s,\textbf{1})\in\mathbb{Z}_p$ for $0\neq s\in \mathbb{Z}_p$. If $\chi$ is a nontrivial primitive Dirichlet character modulo $d$,  with $p^2\nmid d$, then $L_{p}(s,\chi)\in\mathbb{Z}_p$ for $s\in \mathbb{Z}_p$.
   \end{lemma}
\begin{proof}
The first assertion follows directly from the definition. For the second assertion, see \cite[Corollary 5.13]{wa0}.
\end{proof}
The following congruences generalize the  Kummer's congruences for
generalized Bernoulli numbers \cite{wa0}:
\begin{lemma}\label{lemma5} Let  $\chi$ be a nontrivial primitive character modulo $d$, $p^2\nmid d$. Then for  integers $k$, $s$ and $t$, with $0<k<p-2$, we have
\begin{equation}\label{fmain325} L_{p}(s,\chi)\equiv L_{p}(s+p^{k-1}t,\chi) \  (mod\ p^k),\end{equation} and
\begin{equation}\label{fmain325} \Delta_t^k L_{p}(s,\chi)=\sum_{i=0}^k(-1)^i{k\choose
i}L_{p}(s+it,\chi)\equiv 0\  (mod\ p^k),\end{equation} where $\Delta_t$ is the forward difference operator with increment $t$.
   \end{lemma}
\begin{proof}The first congruence follows from \cite[Theorem 5.12]{wa0}. Thus it suffices to prove the second.
We choose $D$ in (\ref{fmain32}) such that $p^2\nmid D$. Then by definition, $L_{p}(s,\chi)$ is an infinite sum of the terms
$$g(s,m,n)=\chi(a)\frac{1}{D}\frac{1}{s-1}{1-s \choose m}{1-s \choose n}(B_{n})(\langle a \rangle-1)^m(\frac{D}{a})^n$$ where $m+n>0$.
 If $m+n>k$, we have $p^k|g(s,m,n)$ for $s\in\mathbb{Z}_p$, hence $p^k|\Delta_t^kg(s,m,n)$. If $m+n\leq k$, then $ g(s,m,n)$ is a polynomial in $s$ of degree less than $k$,
  hence $\Delta_t^kg(s,m,n)=0$.
\end{proof}

For a primitive character $\chi$, and two positive integers $m$, $k$, with $m<p-1$, set
$$B_p(m,k;\chi):=\sum_{i=1}^k(-1)^{i}{k\choose
i}\frac{B_{i(p-1)+1-m,\chi }}{i(p-1)+1-m}, $$ and set
$B_p(m,k)=B_p(m,k;\textbf{1})$. Then by Lemma \ref{lemma5} and
(\ref{fmain321}), if $m+k\leq p-1$ and  $p^2$ does not divide the
conductor of $ \chi\omega^{1-m}$, we have
\begin{equation}\label{fmain327} L_{p}(m,\chi\omega^{1-m})\equiv B_p(m,k;\chi) \equiv -\frac{B_{p^{k-1}(p-1)+1-m,\chi }}{p^{k-1}(p-1)+1-m}\  (mod\ p^k).\end{equation}

 \section{Washington's $p$-adic expansions of sums of powers}
 In\cite{wa}, Washington gave the following $p$-adic expansions of
harmonic sums

 \begin{theorem} \label{main4}
Let $p$ be an odd prime and let $d$, $m$ be positive integers.
Then
\begin{equation}\label{fmain10}\sum_{k=1 \atop
(k,p)=1}^{dp}\frac{1}{k^m}=-\sum_{n=1}^{\infty}{-m\choose
n}L_{p}(m+n,\omega^{1-m-n})(dp)^n.\end{equation}
\end{theorem}

Theorem \ref{main4} together with (\ref{fmain327}), immediately implies the following generalization of Wolstenholme's Theorem:

\begin{corollary} \label{C1}
Let $p$, $d$, and $m$ be as before. Let $j$ be another positive integer, with $j+m\leq p-1$. Then we have
\begin{equation}\label{fmain11}\sum_{k=1 \atop
(k,p)=1}^{dp}\frac{1}{k^m}\equiv -\sum_{n=1}^{j-1}{-m\choose
n}B_p(m+n,j-n)(dp)^n\  (mod\ p^j),\end{equation} and
 \begin{equation}\label{fmain11xia}
 \sum_{k=1 \atop
(k,p)=1}^{dp}\frac{1}{k^m}\equiv \sum_{n=1}^{j-1}-{-m\choose
n}\frac{B_{p^{j-n-1}(p-1)+1-m-n}}{p^{j-n-1}+m+n-1}(dp)^n\  (mod\ p^j).\end{equation}
\end{corollary}

Now we give $p$-adic expansions of ${cp\choose dp}/{c\choose d}$
for $c>d>0$. Set
$$H_p(d;m)=\sum_{k=1 \atop
(k,p)=1}^{dp}\frac{1}{k^m}.$$
\begin{theorem} \label{CC5}
For $c>d>0$, we have
\begin{equation}\label{fmain11xiaqw}
{cp\choose d p}\Big/{c\choose d}=exp\big(-\sum_{k=3}^{\infty}(c^k-(c-d)^k-d^{k})L_{p}(k,\omega^{1-k})p^{k}/k\big).\end{equation}
\end{theorem}
\begin{proof}
By Corollary \ref{C1},
  {\setlength{\arraycolsep}{0pt}
\begin{eqnarray}\label{fmain265511}
&&{cp\choose d p}\Big/{c\choose d}=\prod_{k=1 \atop
(k,p)=1}^{dp}(1+(c-d)p/k)  \nonumber
\\
&&=exp\Big(\sum_{k=1 \atop (k,p)=1}^{dp}log_p(1+\frac{(c-d)p}{k})
\Big)\nonumber
\\
&&=exp\Big(\sum_{k=1 \atop
(k,p)=1}^{dp}\sum_{m=1}^{\infty}(-1)^{m-1}\frac{(c-d)^mp^m}{mk^m}\Big)   \nonumber\\
&&=exp\Big(\sum_{m=1}^{\infty}(-1)^{m-1}H_p(d;m)(c-d)^mp^m/m\Big)    \nonumber\\
&&=exp\Big(\sum_{m=1}^{\infty}\sum_{n=1}^{\infty}(-1)^{m+n}{m+n-1\choose
n}L_{p}(m+n,\omega^{1-m-n})(c-d)^md^np^{m+n}/m\Big)     \nonumber \\
&&=exp\Big(-\sum_{k=3}^{\infty}\sum_{m=1}^{k-1}{k\choose
m}L_{p}(k,\omega^{1-k})(c-d)^md^{k-m}p^{k}/k\Big)
\nonumber   \\
&&=exp\Big(-\sum_{k=3}^{\infty}(c^k-(c-d)^k-d^{k})L_{p}(k,\omega^{1-k})p^{k}/k\Big),
\nonumber
\end{eqnarray}
} where the sixth $=$ follows from the fact that $L_{p}(k,\omega^{1-k})=0$ whenever $k$ is even.
\end{proof}

The above expansion  in terms of  $H_p(d;m)$ or $p$-adic
$L$-functions, together with (\ref{fmain327}), covers many known
congruence about binomial
coefficients\cite{Ba0,ca,gl0,gl1,ht,ka,me,ro,ta,zhao2}. By Theorem
\ref{CC5}, we can easily write down a  $mod\ p^8$ evaluation of
${cp\choose dp}/{c\choose d}$ in terms of Bernoulli numbers:
\begin{corollary} \label{C4}For $p>7$, we have
{\setlength{\arraycolsep}{0pt}
\begin{eqnarray}\label{fmain2655}
&&{cp\choose dp}/{c\choose d}\equiv1-(c^2d-cd^2)\frac{B_{p^{5}-p^4-2}}{p^{4}+2}p^3
\\
&&-(c^4d-2c^3d^2+2c^2d^3-cd^4)\frac{B_{p^{3}-p^2-4}}{p^{2}+4}p^5   \nonumber
\\
&&-(c^6d-3c^5d^2+5c^4d^3-5c^3d^4+3c^2d^5-cd^6)\frac{B_{p-7}}{7}p^7\nonumber\\
&&+(c^2d-cd^2)^2\frac{B_{p^{2}-p-2}^2}{2(p+2)^2}p^6   \  \    (mod\ p^8)
\nonumber
\end{eqnarray}
}\end{corollary}

Next, we show that the $mod\ p^k$ evaluations of the homogeneous multiple harmonic sums(HMHS)
\begin{equation}M_p(d,m;n)=\sum_{1\leq k_1<\cdots< k_n<dp \atop
(k_i,p)=1}\frac{1}{k_1^m\cdots k_n^m},\end{equation} and
\begin{equation}\overline{M}_p(d,m;n)=\sum_{1\leq k_1\leq\cdots\leq k_n<dp }\frac{1}{k_1^m\cdots k_n^m},\end{equation}
can also be reduced to that of $H_p(d;m)$. Let $t$ be an
indeterminate, then formally we have
  {\setlength{\arraycolsep}{0pt}
\begin{eqnarray}\label{fmain2656d}
1+\sum_{n=1}^{dp}M_p(d,m;n)t^n&=&\prod_{k=1 \atop
(k,p)=1}^{dp}(1+\frac{t}{k^m})
\\
&=&exp\Big(\sum_{k=1 \atop
(k,p)=1}^{dp}log(1+\frac{t}{k^m}) \Big)\nonumber
\\
&=&exp\Big(\sum_{k=1 \atop
(k,p)=1}^{dp}\sum_{j=1}^{\infty}(-1)^{j-1}\frac{t^j}{jk^{jm}}\Big)   \nonumber\\
&=&exp\Big(\sum_{j=1}^{\infty}(-1)^{j-1}H_p(d;mj)\frac{t^j}{j}\Big).
\nonumber
\end{eqnarray}
}
Similarly we have
\begin{equation}\label{fmain2657d} 1+\sum_{n=1}^{\infty}\overline{M}_p(d,m;n)t^n=\prod_{k=1 \atop
(k,p)=1}^{dp}(1-\frac{t}{k^m})^{-1}=exp\Big(\sum_{j=1}^{\infty}H_p(d;mj)\frac{t^j}{j}\Big).\end{equation}
Applying Corollary \ref{C1}  to (\ref{fmain2656d}) and
(\ref{fmain2657d}), we get the following congruences which improve
the previous results about HMHS \cite{zhoucai}.
\begin{corollary} \label{C5}
For $n>1$, $p>mn+4,$ we have
\begin{equation*}\label{fmain2658} M_p(d,m;n)\equiv\begin{cases}\sum_{j=1}^{n-1}\frac{m^2(mj+1)B_{p-mj-2}B_{p-m(n-j)-1}}{2(mj+2)(mn-mj+1)}d^3p^3 \\ &\\-\frac{m(mn+1)B_{p^2-p-mn-1}}{2(p+mn+1)}d^2p^2\ \ \ \ (mod\ p^4)&\text{if}\ mn\  is\ odd \\ &\\ (-1)^{n}\sum_{j=1}^{n-1}\frac{m^2B_{p-mj-1}B_{p-m(n-j)-1}}{2(mj+1)(mn-mj+1)}d^2p^2 \\ &\\
-(-1)^{n}\frac{mB_{p^2-p-mn}}{p+mn}dp \ \ \ \   \ \ \  \ \ \ (mod\ p^3)& \text{if}\ m\ is\ even
\\ &\\\sum_{j=1}^{n-1}\frac{m^2B_{p-mj-1}B_{p-m(n-j)-1}}{2(mj+1)(mn-mj+1)}d^2p^2 \\ &\\
-\frac{mB_{p^2-p-mn}}{p+mn}dp  \  \ \ \  \ \ \  \ \ \  \ \ \  \  \ \ \ (mod\ p^3)& \text{if}\ m\ is\ odd\ and\ n\ is\ even
\\ \end{cases}\end{equation*}
and
\begin{equation*}\label{fmain2658} \overline{M}_p(d,m;n)\equiv\begin{cases}-\sum_{j=1}^{n-1}\frac{m^2(mj+1)B_{p-mj-2}B_{p-m(n-j)-1}}{2(mj+2)(mn-mj+1)}d^3p^3 \\ &\\-\frac{m(mn+1)B_{p^2-p-mn-1}}{2(p+mn+1)}d^2p^2\ \ \ \ (mod\ p^4)&\text{if}\ mn\  is\ odd \\ &\\ \sum_{j=1}^{n-1}\frac{m^2B_{p-mj-1}B_{p-m(n-j)-1}}{2(mj+1)(mn-mj+1)}d^2p^2 \\ &\\
+\frac{mB_{p^2-p-mn}}{p+mn}dp \ \ \ \   \ \ \  \ \ \ (mod\ p^3)& \text{if}\ m\ is\ even
\\ &\\\sum_{j=1}^{n-1}\frac{m^2B_{p-mj-1}B_{p-m(n-j)-1}}{2(mj+1)(mn-mj+1)}d^2p^2 \\ &\\
-\frac{mB_{p^2-p-mn}}{p+mn}dp  \  \ \ \  \ \ \  \ \ \  \ \ \  \  \ \ \ (mod\ p^3)& \text{if}\ m\ is\ odd\ and\ n\ is\ even
\\ \end{cases}\end{equation*}
\end{corollary}
 \section{$p$-adic expansions of sums of
 Lemher's type}

This section is parallel to the previous section. First we generalize Washington's $p$-adic expansions to cover the sums of
 Lemher's type.
 \begin{theorem} \label{main4k}
Let $d$, $r$, $m$ be positive integers, with $(r, dp)=1$.
Then, when $m=1$,{\setlength{\arraycolsep}{0pt}
\begin{eqnarray}\label{fmain100sd}
&&\sum_{k=1 \atop
(k,p)=1}^{[dp/r]}\frac{1}{k}=-\frac{1}{p}\big(log_pr^{p-1}+\sum_{q|r \atop q\
prime}\frac{log_pq^{p-1}}{q-1}\big)-
\\
&&\frac{r}{\varphi(r)} \sum_{n=0 \atop f_{\psi}|r (\psi,n)\neq(\textbf{1},0)}^{\infty}\overline{\psi}(-dp)(-1)^{
n}d(\psi,r,-1-n)L_{p}(1+n,\psi\omega^{-n})(dp)^n,\nonumber
\end{eqnarray}
}and, when $m>1$,
{\setlength{\arraycolsep}{0pt}
\begin{eqnarray}\label{fmain100se}
&&\sum_{k=1 \atop
(k,p)=1}^{[dp/r]}\frac{1}{k^{m}}=L_{p}(m,\omega^{1-m})-
\\
&&\frac{r^{m}}{\varphi(r)}\sum_{{f_{\psi}| r}}\overline{\psi}(-dp)\sum_{n=0}^{\infty}{-m\choose
n}d(\psi,r,-m-n)L_{p}(m+n,\psi\omega^{1-m-n})(dp)^n,\nonumber
\end{eqnarray}
}where $f_{\psi}$ is the conductor of primitive character $\psi$,  and
$$d(\psi,r,n) =\prod_{q\mid \frac{r}{f_{\psi}}\atop q\
prime}(1-\psi(q)q^{n}).$$
\end{theorem}
\begin{proof}
Note that the term $-\frac{1}{p}\big(log_pr^{p-1}+\sum_{q|r \atop
q\ prime}\frac{log_pq^{p-1}}{q-1}\big)$ in (\ref{fmain100sd}) is
just the value of $$\Big(1-r^{s(p-1)}\prod_{q|r \atop q\
prime}\frac{1-q^{-1-s(p-1)}}{1-1/q}\Big)/(s(p-1))$$ at $s=0$.
Hence it suffices to prove the following identity for $m\geq1$ and
$s\in\mathbb{Z}_p$: {\setlength{\arraycolsep}{0pt}
\begin{eqnarray}\label{fmain100}
&&\sum_{k=1 \atop
(k,p)=1}^{[dp/r]}\frac{1}{k^{m+s(p-1)}}=L_{p}(m+s(p-1),\omega^{1-m}) \nonumber
\\
&&-\frac{r^{m+s(p-1)}}{\varphi(r)}\sum_{{f_{\psi}| r}}\overline{\psi}(-dp)\sum_{n=0}^{\infty}{-m-s(p-1)\choose
n}d(\psi,r,-m-s(p-1)-n) \nonumber
\\
&&L_{p}(m+s(p-1)+n,\psi\omega^{1-m-n})(dp)^n,\nonumber \\
&&=(m+s(p-1)-1)L_{p}(m+s(p-1),\omega^{1-m})\frac{\big(1-r^{m+s(p-1)-1}\prod_{q|r \atop q\
prime}\frac{1-q^{-m-s(p-1)}}{1-1/q}\big)}{(m+s(p-1)-1)} \nonumber
\\
&&-\frac{r^{m+s(p-1)}}{\varphi(r)} \sum_{n=0 \atop f_{\psi}|r (\psi,n)\neq(\textbf{1},0)}^{\infty}\overline{\psi}(-dp){-m-s(p-1)\choose
n}d(\psi,r,-m-s(p-1)-n) \nonumber
\\
&&L_{p}(m+s(p-1)+n,\psi\omega^{1-m-n})(dp)^n,\nonumber
\end{eqnarray}
} where $a^{s(p-1)}=(a^{p-1})^s$. It is easy to see that both
sides are analytic on $\mathbb{Z}_p$. We first prove the case
$m+s(p-1)$ is negative integer, and the theorem follows by
continuity. We note that for $n>0$,
$${n-1-s(p-1)\choose
n}L_{p}(1+s(p-1),\textbf{1})|_{s=0}=-\frac{1}{n}(1-\frac{1}{p}).$$
By Lemma \ref{lemma1}, when $m<0$,
  {\setlength{\arraycolsep}{0pt}
\begin{eqnarray}\label{fmain2655}
&&\sum_{k=1 \atop
(k,p)=1}^{[dp/r]}\frac{1}{k^m}\nonumber
\\
&=&\sum_{k=1}^{[dp/r]}\frac{1}{k^m}-\sum_{k=1}^{[d/r]}\frac{1}{p^mk^m} \nonumber
\\
&=&-(1-p^{-m})\frac{B_{-m+1}}{-m+1}+\frac{r^m}{\varphi(r)}d(\textbf{1},r,-1)\big(\frac{1-\frac{1}{p}}{-m+1}\big)(dp)^{-m+1}\nonumber
\\
&+&\frac{r^m}{\varphi(r)}\sum_{\psi\neq\textbf{1} \atop f_{\psi}| r}\overline{\psi}(-dp)\sum_{n=0}^{-m}{-m\choose
n}d(\psi,r,-m-n) (1-\psi(p)p^{-m-n})\frac{B_{-m-n+1,\psi}}{-m-n+1}c  \nonumber\\
&=&L_{p}(m,\omega^{1-m})-\frac{r^m}{\varphi(r)}d(\textbf{1},r,-1)(dp)^{-m+1}{-m-s(p-1)\choose
-m+1}L_{p}(1+s(p-1),\textbf{1})|_{s=0} \nonumber
\\
&-&\frac{r^m}{\varphi(r)}\sum_{f_{\psi}| r}\overline{\psi}(-dp)\sum_{n=0}^{-m}{-m\choose
n}d(\psi,r,-m-n)L_{p}(m+n,\psi\omega^{1-m-n})(dp)^n.  \nonumber
\end{eqnarray}
} Hence the theorem follows.
\end{proof}

Combining Theorem \ref{main4k} with  (\ref{fmain327}), implies the following generalization of Lemher's congruences:

\begin{corollary} \label{C2}
Let $p$, $d$, $r$ and $m$ be as before. Assume that $(\varphi(r),
p)=1$. Let $j$ be another positive integer, with $j+m\leq p-1$.
Then, when $m=1$,
 {\setlength{\arraycolsep}{0pt}
\begin{eqnarray}\label{fmain100}
\sum_{k=1 \atop
(k,p)=1}^{[dp/r]}\frac{1}{k}&\equiv&-\frac{r}{\varphi(r)}\sum_{n=0,\atop f_{\psi}| r, (\psi,n)\neq(\textbf{1},0)}^{j-1}\overline{\psi}(-dp)d(\psi,r,-1-n)B_p(1+n,j-n;\psi)(-dp)^n       \nonumber
\\
&-& \frac{1}{p}\big(log_pr^{p-1}+\sum_{q|r \atop q\
prime}\frac{log_pq^{p-1}}{q-1}\big)\  \   (mod\ p^j),\nonumber
\end{eqnarray}
}
and when $m>1$,
 {\setlength{\arraycolsep}{0pt}
\begin{eqnarray}\label{fmain100}
\sum_{k=1 \atop
(k,p)=1}^{[dp/r]}\frac{1}{k^m}&\equiv&-\frac{r^m}{\varphi(r)}\sum_{f_{\psi}| r}\overline{\psi}(-dp)\sum_{n=0}^{j-1}{-m\choose
n}d(\psi,r,-m-n)B_p(m+n,j-n;\psi)(dp)^n\nonumber
\\
&+&B_p(m,j;\textbf{1}) \  \  (mod\ p^j).\nonumber
\end{eqnarray}
}The above congruences remain valid if we replace some
$B_p(m,k;\chi)$ with $-\frac{B_{p^{k-1}(p-1)+1-m,\chi
}}{p^{k-1}(p-1)+1-m}$.\end{corollary}

Let $E_n$(Euler numbers) be defined by
\begin{equation}\label{fmain15} \frac{2e^{z}}{e^{2z}+1}=\sum_{n=0}^{\infty}E_nz^n/n!.\end{equation}
Then it is easy to see that
\begin{equation}\frac{B_{n+1,\eta}}{n+1}=-\frac{E_{n}}{2}=\begin{cases} 2^{2n+1}\frac{B_{n+1}(1/4)}{n+1}& \text{if}\ n\ is\ even \\
0& \text{if}\ n\ is\ odd \end{cases}\end{equation}
where  $\eta$ is the unique quadratic character module $4$. Let $r=2$, or $4$ in  Corollary \ref{C2}, we get
\begin{corollary} \label{Cj2}
Assume that $(2, d)=1$ and $j+m\leq p-1$. Then, when $m>1$, we have

 {\setlength{\arraycolsep}{0pt}
\begin{eqnarray}\label{fmain26575}
\sum_{k=1 \atop
(k,p)=1}^{[dp/2]}\frac{1}{k^m}&\equiv& \frac{B_{p^{j-1}(p-1)+1-m}}{p^{j-1}+m-1}\nonumber \\
&-&\sum_{n=0}^{j-1}(2^m-2^{-n}){-m\choose
n}\frac{B_{p^{j-n-1}(p-1)+1-m-n}}{p^{j-n-1}+m+n-1}(dp)^n\  (mod\ p^j), \nonumber
\end{eqnarray}}
and
 {\setlength{\arraycolsep}{0pt}
\begin{eqnarray}\label{fmain26575}
\sum_{k=1 \atop
(k,p)=1}^{[dp/4]}\frac{1}{k^m}&\equiv&\frac{B_{p^{j-1}(p-1)+1-m}}{p^{j-1}+m-1}\nonumber \\
&-&\sum_{n=0}^{j-1}(2^{2m-1}-2^{m-n-1}){-m\choose
n}\frac{B_{p^{j-n-1}(p-1)+1-m-n}}{p^{j-n-1}+m+n-1}(dp)^n\nonumber \\
&+&(-1)^{\tfrac{dp-1}{2}}2^{2m-2}\sum_{n=0}^{j-1}{-m\choose
n}E_{p^{j-n-1}(p-1)-m-n}(dp)^n\  (mod\ p^j).\nonumber
\end{eqnarray}}
When $m=1$, we have
 {\setlength{\arraycolsep}{0pt}
\begin{eqnarray}\label{fmain26574}
 \sum_{k=1 \atop
(k,p)=1}^{[dp/2]}\frac{1}{k}&\equiv&-\frac{2 log_p2^{p-1}}{p}-\sum_{n=1}^{j-1}(2-2^{-n})\frac{B_{p^{j-n-1}(p-1)-n}}{p^{j-n-1}+n}(-dp)^n  \nonumber  \\
&\equiv&  2\sum_{n=1}^{j}(-1)^n\frac{q_p(2)^{n}p^{n-1}}{n} \nonumber  \\
&-&\sum_{n=1}^{j-1}(2-2^{-n})\frac{B_{p^{j-n-1}(p-1)-n}}{p^{j-n-1}+n}(-dp)^n \  (mod\ p^j),\nonumber
\end{eqnarray}}
and  {\setlength{\arraycolsep}{0pt}
\begin{eqnarray}\label{fmain26574}
\sum_{k=1 \atop
(k,p)=1}^{[dp/4]}\frac{1}{k}&\equiv&-\frac{3 log_p2^{p-1}}{p}-\sum_{n=1}^{j-1}(2-2^{-n})\frac{B_{p^{j-n-1}(p-1)-n}}{p^{j-n-1}+n}(-dp)^n \nonumber \\
&-&(-1)^{\tfrac{dp-1}{2}}\sum_{n=0}^{j-1}E_{p^{j-n-1}(p-1)-n-1}(dp)^n  \nonumber  \\
&\equiv&  3\sum_{n=1}^{j}(-1)^n\frac{q_p(2)^{n}p^{n-1}}{n}-\sum_{n=1}^{j-1}(2-2^{-n})\frac{B_{p^{j-n-1}(p-1)-n}}{p^{j-n-1}+n}(-dp)^n \nonumber  \\
&-&(-1)^{\tfrac{dp-1}{2}}\sum_{n=0}^{j-1}E_{p^{j-n-1}(p-1)-n-1}(dp)^n\  (mod\ p^j);\nonumber
\end{eqnarray}}
\end{corollary}

Let $p$, $d$, $r$ be as before and let $c$ be a nonzero integer. Set
$$H_p(d,r;m)=\sum_{k=1 \atop
(k,p)=1}^{[dp/r]}\frac{1}{k^m}.$$
As in Section 3, we have the following  $p$-adic expansion of ${cp+[dp/r]\choose [dp/r]}/{c+[d/r]\choose [d/r]}$
 \begin{theorem} \label{main1hz}
 {\setlength{\arraycolsep}{0pt}
\begin{eqnarray}\label{fmain2655}
&&{cp+[dp/r]\choose [dp/r]}\Big/{c+[d/r]\choose [d/r]}=
=exp\Big(\sum_{m=1}^{\infty}(-1)^{m-1}H_p(d,r;m)c^mp^m/m\Big)
\nonumber
\\
&&=exp\Big(-c log_pr^{p-1}-c\sum_{q|r \atop q\
prime}\frac{log_pq^{p-1}}{q-1}\Big)\cdot
\nonumber   \\
&&exp\big(\sum_{k=3}^{\infty}L_{p}(k,\omega^{1-k})c^{k}p^{k}/k\big)\cdot
\nonumber   \\
&&exp\big(\frac{-1}{\varphi(r)}\sum_{k=3}^{\infty}((cr+d)^k-d^{k})d(\textbf{1},r,-k)L_{p}(k,\omega^{1-k})p^{k}/k)\big) \cdot
\nonumber   \\
 &&exp\big(\frac{-1}{\varphi(r)}\sum_{\psi\neq\textbf{1} \atop f_{\psi}| r}\overline{\psi}(dp)\sum_{k=1}^{\infty}((cr+d)^k-d^{k})d(\psi,r,-k)L_{p}(k,\psi\omega^{1-k})p^{k}/k)\big)     \nonumber
\end{eqnarray}
} \end{theorem}
The above expansion, together with (\ref{fmain327}), immediately implies many known congruence about ${p-1\choose [p/r]}$ ($r>1$) \cite{ca,pan,sun7,sun8}. For example, setting  $c=-1$, $d=1$, $r=2$, or $4$ in Theorem \ref{main1hz}, we get:
\begin{corollary} \label{Cs1}
For $p>7$, {\setlength{\arraycolsep}{0pt}
\begin{eqnarray}\label{fmainxz2655}
&&(-1)^{\frac{p-1}{2}}4^{-p+1}{p-1\choose
\frac{p-1}{2}}\equiv1+\frac{1}{4}\frac{B_{p^{5}-p^4-2}}{p^{4}+2}p^3
+\frac{3}{16}\frac{B_{p^{3}-p^2-4}}{p^{2}+4}p^5    \nonumber
\\
&&+\frac{1}{16}\frac{B_{p^{2}-p-2}^2}{2(p+2)^2}p^6
+\frac{9}{64}\frac{B_{p-7}}{7}p^7  \  \    (mod\ p^8).  \nonumber
\end{eqnarray}
}  \end{corollary}

\begin{corollary} \label{Cs1v}
For $p>5$, {\setlength{\arraycolsep}{0pt}
\begin{eqnarray}\label{fmainxz2655}
&&(-1)^{[\frac{p}{4}]}2^{-3p+3}{p-1\choose
[\frac{p}{4}]}\equiv1-(-1)^{\frac{p-1}{2}}E_{p^{4}-p^3-2}p^2+\frac{15}{4}\frac{B_{p^{3}-p^2-2}}{p^{2}+2}p^3
\nonumber
\\
&& -(-1)^{\frac{p-1}{2}}5E_{p^{2}-p-4}p^4 +\frac{75}{16}B_{p-5}p^5   \nonumber
\\
&& \frac{1}{2}(2E_{p-3}-E_{2p-4})^2p^4-(-1)^{\frac{p-1}{2}}\frac{5}{4}E_{p-3}B_{p-3}p^5 \  \    (mod\ p^6).  \nonumber
\end{eqnarray}
}  \end{corollary}

Finally, we consider the $mod\ p^k$ evaluations of the homogeneous
multiple harmonic sums of Lemher's type
$$M_p(d,r,m;n)=\sum_{1\leq k_1<\cdots< k_n<dp/r \atop
(k_i,p)=1}\frac{1}{k_1^m\cdots k_n^m},$$ and
$$\overline{M}_p(d,r,m;n)=\sum_{1\leq k_1\leq\cdots\leq k_n<dp/r
\atop (k_i,p)=1}\frac{1}{k_1^m\cdots k_n^m}.$$ As before we have
\begin{equation}\label{fmain2656} 1+\sum_{n=1}^{[dp/r]}M_p(d,r,m;n)t^n=exp(\sum_{j=1}^{\infty}(-1)^{j-1}H_p(d,r;mj)\frac{t^j}{j}),\end{equation}
and
\begin{equation}\label{fmain2657} 1+\sum_{n=1}^{\infty}\overline{M}_p(d,r,m;n)t^n=exp(\sum_{j=1}^{\infty}H_p(d,r;mj)\frac{t^j}{j}).\end{equation}
Applying Corollary \ref{C2} to (\ref{fmain2656}) and
(\ref{fmain2657}) we can deduce $mod\ p^k$ evaluations of
$M_p(d,r,m;n)$ and $\overline{M}_p(d,r,m;n)$ for any $k$. The
following are two examples:
\begin{corollary} \label{C5ug}
Assume that $p>mn+4,$ and $(d,2)=1$, then when $m>1$, we have
\begin{equation*}\label{fmain2658} M_p(d,2,m;n)\equiv\begin{cases}\sum_{j=1}^{n-1}\frac{(2^{mj}-2)(2^{m(n-j)}-\frac{1}{2})B_{p-mj}B_{p-m(n-j)-1}}{j^2(mn-mj+1)}dp \\
&\\ -\frac{(2^{mn}-2)B_{p^2-p-mn+1}}{n(p+mn-1)}\ \ \ \ (mod\ p^2)&\text{if}\ mn\  is\ odd \\
 &\\ (-1)^{n}\sum_{j=1}^{n-1}\frac{m^2(2^{mj}-\frac{1}{2})(2^{m(n-j)}-\frac{1}{2})B_{p-mj-1}B_{p-m(n-j)-1}}{2(mj+1)(mn-mj+1)}d^2p^2 \\ &\\
-(-1)^{n}\frac{m(2^{mn}-\frac{1}{2})B_{p^2-p-mn}}{p+mn}dp \ \ \ \
\ \ \  \ \ \ (mod\ p^3)& \text{if}\ m\ is\ even
\\ \end{cases}\end{equation*}
and, when $m=1$, and $n>1$ is odd we have
{\setlength{\arraycolsep}{0pt}
\begin{eqnarray}\label{fmainxz2655}
&&M_p(d,2,1;n)\equiv\sum_{j=2}^{n-2}\frac{(2^{j}-2)(2^{n-j}-\frac{1}{2})B_{p-j}B_{p-n+j-1}}{j^2(n-j+1)}dp
\ \    \nonumber
\\
&& +\frac{q_p(2)(2^{n}-1)B_{p-n}}{n}dp-\frac{(2^{n}-2)B_{p^2-p-n+1}}{n(p+n-1)}\ \ \ \ (mod\ p^2).
\nonumber
\end{eqnarray}
}
\end{corollary}

\section{ congruences of Gauss and Jacobi}
In this section, we will give $p$-adic expansions of ${(p-1)/2\choose (p-1)/4}$ (when $4|p-1$) and ${2(p-1)/3\choose (p-1)/3}$ (when $3|p-1$), and hence full generalizations of congruences of Gauss and Jacobi. First we introduce Morita's $p$-adic gamma function.
For a positive integer $k$, set \begin{equation}\label{fmain2659}  \Gamma_p(k)=(-1)^k\prod_{j=1 \atop
(j,p)=1}^{k-1}j.\end{equation} Then $\Gamma_p$ extends uniquely to a continuous function from $\mathbb{Z}_p$ to $\mathbb{Z}_p^*$\cite{ko}.
Let $a$, $b,$ and $m$ be positive integers satisfying $p\equiv1 \ (mod\ m)$, and  $a+b\leq m$. Then we have
{\setlength{\arraycolsep}{0pt}
\begin{eqnarray}\label{fmainza2655}
&&{(a+b)(p-1)/m\choose a(p-1)/m}=
\frac{{(a+b)(p^k-1)/m\choose a(p^k-1)/m}{a(p^k-1)/m\choose a(p-1)/m}{b(p^k-1)/m\choose b(p-1)/m}}
{{(a+b)(p^k-1)/m\choose (a+b)(p-1)/m}{(a+b)(p^k-p)/m\choose a(p^k-p)/m}}
\\
&&=\frac{-\Gamma_p(1-\frac{a+b}{m})}{\Gamma_p(1-\frac{a}{m})\Gamma_p(1-\frac{b}{m})}
\lim_{k\rightarrow\infty} \frac{{a(p^k-1)/m\choose
a(p-1)/m}{b(p^k-1)/m\choose b(p-1)/m}{(a+b)(p^{k-1}-1)/m\choose
a(p^{k-1}-1)/m}} {{(a+b)(p^k-1)/m\choose
(a+b)(p-1)/m}{(a+b)(p^k-p)/m\choose a(p^k-p)/m}} \nonumber
\end{eqnarray}
} By (\ref{fmain265511}) and (\ref{fmain2655}), the above limit
has a $p$-adic expansion. Now assume that $4|p-1$ and let $a$ be
as in congruence (\ref{fmain1fd}).
 \begin{theorem} \label{mainx41}
{\setlength{\arraycolsep}{0pt}
\begin{eqnarray}\label{fmaince2655}
&&{(p-1)/2\choose (p-1)/4}=
\Big(2a-2a\sum_{j=1}^{\infty}\frac{1}{j}{2j-2\choose
j-1}(\frac{p}{4a^2})^j\Big)  \nonumber
\\
&&exp\Big(\frac{1}{2}
log_p2^{p-1}+\sum_{k=2}^{\infty}L_{p}(k,\eta\omega^{1-k})p^{k}/k)\Big),
\nonumber
\end{eqnarray}
} where  $\eta$ is the unique quadratic character module $4$.
\end{theorem}
\begin{proof}
We need the following expansion from \cite[Corollary 1]{cd}
\begin{equation}\label{fmain2659} \frac{\Gamma_p(1-\frac{1}{2})}{\Gamma_p(1-\frac{1}{4})^2}=-2a+2a\sum_{j=1}^{\infty}\frac{1}{j}{2j-2\choose j-1}(\frac{p}{4a^2})^j.\end{equation}

By Theorem \ref{CC5}, we have {\setlength{\arraycolsep}{0pt}
\begin{eqnarray}\label{fmaince2655}
&&\lim_{k\rightarrow\infty} \frac{{(p^k-p)/2\choose (p^k-p)/4}}
{{(p^{k-1}-1)/2\choose
(p^{k-1}-1)/4}}=exp\Big(\sum_{k=3}^{\infty}(2^{-k}-2^{-2k+1})L_{p}(k,\omega^{1-k})p^{k}/k)\Big),
\end{eqnarray}
} and by Theorem \ref{main1hz},
{\setlength{\arraycolsep}{0pt}
\begin{eqnarray}\label{fmaince2655}
&&\lim_{k\rightarrow\infty} \frac{{(p^k-1)/4\choose (p-1)/4}^2}
{{(p^k-1)/2\choose (p-1)/2}}=exp\Big(\frac{1}{2}
log_p2^{p-1}+\sum_{k=3}^{\infty}(2^{-k}-2^{-2k+1})L_{p}(k,\omega^{1-k})p^{k}/k\Big)\cdot
\nonumber   \\
 &&exp\Big(\sum_{k=2}^{\infty}L_{p}(k,\eta\omega^{1-k})p^{k}/k)\Big)
\nonumber
\end{eqnarray}
}Hence
{\setlength{\arraycolsep}{0pt}
\begin{eqnarray}\label{fmaince2655}
&&\lim_{k\rightarrow\infty} \frac{{(p^k-1)/4\choose
(p-1)/4}^2{(p^{k-1}-1)/2\choose (p^{k-1}-p)/4}} {{(p^k-1)/2\choose
(p-1)/2}{(p^k-p)/2\choose (p^k-p)/4}}=exp\Big(\frac{1}{2}
log_p2^{p-1}+\sum_{k=2}^{\infty}L_{p}(k,\eta\omega^{1-k})p^{k}/k)\Big),
\nonumber
\end{eqnarray}
}and the theorem follows.
\end{proof}

Now assume that $3|p-1$ and let $a$ be as in congruence
(\ref{fmain1fd1}). Using exactly the same proof, we get
 \begin{theorem} \label{mainx4}
{\setlength{\arraycolsep}{0pt}
\begin{eqnarray}\label{fmaince2655}
&&{2(p-1)/3\choose (p-1)/3}=
\Big(-r+r\sum_{j=1}^{\infty}\frac{1}{j}{2j-2\choose
j-1}(\frac{p}{r^2})^j\Big)\cdot  \nonumber
\\
&&exp\Big(\sum_{k=3}^{\infty}(1-2^{k-1})(1-3^{-k})L_{p}(k,\omega^{1-k})p^{k}/k\Big)\cdot
\nonumber   \\
 &&exp\Big(\sum_{k=2}^{\infty}(1+2^{k-1})L_{p}(k,\phi\omega^{1-k})p^{k}/k)\Big),
\nonumber
\end{eqnarray}
} where  $\phi$ is the unique quadratic character module $3$.
\end{theorem}

\begin{remark}Using the $p$-adic expansions of
$\frac{-\Gamma_p(1-\frac{a+b}{m})}{\Gamma_p(1-\frac{a}{m})\Gamma_p(1-\frac{b}{m})}$
(m=4,6,8) obtained in \cite{ad}, we can give similar expansions
for ${(a+b)(p-1)/m\choose a(p-1)/m}$ (m=4,6,8).
\end{remark}

 \section{Wilson's theorem and related congruences}
In this section, we show how to  get the $mod\ p^k$ evaluations  of $(p-1)!$, and  $(\frac{p-1}{2})!$ (when $4|p-3$) and  $(\frac{p-1}{4})!^4$ (when $4|p-1$),
 and hence full generalizations of Wilson's theorem and related congruences.
 By (\ref{fmain2659xsws}), we have for $m<p$,  {\setlength{\arraycolsep}{0pt}
\begin{eqnarray}\label{fmain265cs5}
&&log_p(-(p-1)!)=\frac{1}{p-1}log_p(p-1)!^{p-1}
\\
&&=\frac{1}{p-1}\sum_{k=1}^{p-1}log_p(1+(k^{p-1}-1))   \nonumber\\
&&=\frac{1}{p-1}\sum_{k=1}^{p-1}\sum_{n=1}^{\infty}(-1)^{n-1}(k^{p-1}-1)^n/n
\nonumber   \\
&&\equiv\frac{1}{p-1}\sum_{n=1}^{m-1}\frac{-1}{n}\sum_{j=0}^{n}(-1)^j{n\choose
j} \sum_{k=1}^{p-1}k^{j(p-1)}
\nonumber   \\
&&\equiv\frac{1}{p-1}\sum_{n=1}^{m-1}\frac{1}{n}
\Big(\sum_{j=1}^{n}\frac{(-1)^{j-1}{n\choose j}}{j(p-1)+1}
\sum_{k=0}^{j(p-1)}{j(p-1)+1\choose k+1}B_{j(p-1)-k}p^{k+1}\Big)
\nonumber \\ &&+\sum_{n=1}^{m-1}\frac{-1}{n} \    (mod\ p^m)
\nonumber
\end{eqnarray}
}

From the above expansion, we can deduce $mod\ p^m$ evaluations  of
$(p-1)!$ for any $m>0$. But when $m>4$, the congruence would be
 too complicated to be written down. Thus we will only work out the case $m=4$. Set $V_{p,i}=\frac{pB_{i(p-1)}}{p-1}-1$.
\begin{lemma}\label{lemma4h9}
 \begin{equation}\label{fmain2659xsn}  V_{p,i}\equiv 0\ (mod\ p),\end{equation}
\begin{equation}\label{fmain2659xs}  2V_{p,1}-V_{p,2}\equiv 3V_{p,2}-2V_{p,3}\equiv 0\ (mod\ p^2),\end{equation}  and
  \begin{equation}\label{fmain2659xs1}  3V_{p,1}-3V_{p,2}+V_{p,3}\equiv 0\ (mod\ p^3).\end{equation}
   \end{lemma}
\begin{proof}
The first congruence follows  from the von Staudt-Clausen theorem. The second and third follow by evaluating $\sum_{k=1}^{p-1}(k^{p-1}-1)^2$  modulo $p^2$, and $\sum_{k=1}^{p-1}(k^{p-1}-1)^3$ modulo $p^3$.
\end{proof}
    \begin{theorem} \label{mainx4} For $p>3$,
{\setlength{\arraycolsep}{0pt}
\begin{eqnarray}\label{fmaince2655}
(p-1)!&\equiv&-1-3V_{p,1}+\frac{3}{2}V_{p,2}-\frac{1}{3}V_{p,3}-\frac{1}{2}(2V_{p,1}-\frac{1}{2}V_{p,2})^2- \frac{1}{6}V_{p,1}^3 \nonumber
\\
&&+(B_{p-3}-\frac{3}{2}B_{2p-4}+\frac{2}{3}B_{3p-5})p^3 \ (mod\ p^4),
\nonumber
\end{eqnarray}
}
 {\setlength{\arraycolsep}{0pt}
\begin{eqnarray}\label{fmaince2655}
(-1)^{\frac{p-1}{2}}4^{p-1}(\frac{p-1}{2})!^2&\equiv&-1-3V_{p,1}+\frac{3}{2}V_{p,2}-\frac{1}{3}V_{p,3}
-\frac{1}{2}(2V_{p,1}-\frac{1}{2}V_{p,2})^2-\frac{1}{6}V_{p,1}^3 \nonumber
\\
&&+(\frac{13}{12}B_{p-3}-\frac{3}{2}B_{2p-4}+\frac{2}{3}B_{3p-5})p^3 \ (mod\ p^4),
\nonumber
\end{eqnarray}
}
\end{theorem}
\begin{proof}
by (\ref{fmain265cs5}) we have {\setlength{\arraycolsep}{0pt}
\begin{eqnarray}\label{fmaince2655}
(p-1)!&\equiv&-exp(3V_{p,1}-\frac{3}{2}V_{p,2}+\frac{1}{3}V_{p,3})\cdot  \nonumber
\\
&&exp\big((\frac{p-2}{2}B_{p-3}-\frac{2p-3}{2}B_{2p-4}+\frac{3p-4}{6}B_{3p-5})p^3\big)\ (mod\ p^4).
\nonumber
\end{eqnarray}
}
By (\ref{fmain2659xs}) and (\ref{fmain2659xs1}),
 \begin{equation}\label{fmain2659x2s}  (3V_{p,1}-\frac{3}{2}V_{p,2}+\frac{1}{3}V_{p,3})^2 \equiv (2V_{p,1}-\frac{1}{2}V_{p,2})^2\ (mod\ p^4),\end{equation}
 and
 \begin{equation}\label{fmain2659x22s}  (3V_{p,1}-\frac{3}{2}V_{p,2}+\frac{1}{3}V_{p,3})^3 \equiv V_{p,1}^3\ (mod\ p^4).\end{equation}
 Hence the first assertion follows. The second follows from the first and Corollary \ref{Cs1}.
\end{proof}

\begin{corollary} \label{Cs1vd} If $4|p-3$,
{\setlength{\arraycolsep}{0pt}
\begin{eqnarray}\label{fmaince2655}
2^{p-1}(\frac{p-1}{2})!&\equiv&(-1)^N\Big(1+\frac{3}{2}V_{p,1}-\frac{3}{4}V_{p,2}+\frac{1}{6}V_{p,3}+\frac{1}{8}(2V_{p,1}-\frac{1}{2}V_{p,2})^2+ \frac{1}{48}V_{p,1}^3 \nonumber
\\
&&-(\frac{13}{24}B_{p-3}-\frac{3}{4}B_{2p-4}+\frac{1}{3}B_{3p-5})p^3\Big) \ (mod\ p^4)
\nonumber
\end{eqnarray}
} where $N$ is the number of quadratic nonresidues less than $p/2$.
  \end{corollary}

We close the paper by noting that, when $4|p-1$, we can deduce a congruence for
$(\frac{p-1}{4})!^4$ from Theorems \ref{mainx41} and \ref{mainx4}.

\end{document}